\def\todaysdate{31\textsuperscript{st} January 2022}
\documentclass[reqno,a4paper]{article}
\usepackage{etex}
\usepackage[T1]{fontenc}
\usepackage[utf8]{inputenc}
\usepackage{lmodern}

\usepackage[style=alphabetic,firstinits=true,backend=biber,isbn=false,url=false,maxbibnames=99,maxalphanames=4,texencoding=ascii]{biblatex}

\renewbibmacro{in:}{}
\newbibmacro{string+doi}[1]{\iffieldundef{doi}{\iffieldundef{url}{#1}{\href{\thefield{url}}{#1}}}{\href{http://dx.doi.org/\thefield{doi}}{#1}}}
\DeclareFieldFormat*{title}{\usebibmacro{string+doi}{\emph{#1}}}
\DefineBibliographyStrings{english}{%
  backrefpage = {$\uparrow$},
  backrefpages = {$\uparrow$},
}

\usepackage{amssymb,amsmath,amsthm}
\usepackage{thmtools,xcolor}
\usepackage{units}
\usepackage{graphicx}
\usepackage[all]{xy}
\usepackage{tikz}
\usetikzlibrary{arrows,calc,positioning,decorations.pathreplacing}
\usepackage{tikz-cd}
\usepackage{relsize}
\usepackage{mhsetup}
\usepackage{mathtools}
\usepackage{stmaryrd}
\usepackage{tocloft}
\usepackage{paralist} % For compactitem environment
\usepackage[left=3cm,right=3cm,top=3cm,bottom=3cm]{geometry} % Decrease the margins %
\usepackage[bottom]{footmisc}
\usepackage[colorlinks,citecolor=blue,linkcolor=blue,urlcolor=blue,filecolor=blue,breaklinks]{hyperref}
\usepackage{footnotebackref}
\usepackage[format=plain,indention=1em,labelsep=quad,labelfont={up},textfont={footnotesize},margin=2em]{caption}
\usepackage[mathscr]{euscript}
\usepackage{accents}
\usepackage{xparse}
\usepackage{array}
\usepackage{multirow}

\definecolor{lightblue}{rgb}{0.8,0.8,1}

\definecolor{vdarkred}{rgb}{0.7,0,0}
\declaretheoremstyle[
  spaceabove=\topsep,
  spacebelow=\topsep,
  headpunct=,
  numbered=no,
  postheadspace=1ex,
  headfont=\color{vdarkred}\normalfont\bfseries,
  bodyfont=\normalfont\itshape,
]{colored}
\declaretheoremstyle[
  spaceabove=\topsep,
  spacebelow=\topsep,
  headpunct=,
  numbered=no,
  postheadspace=1ex,
  headfont=\normalfont\bfseries,
  bodyfont=\normalfont\itshape,
]{italic}
\declaretheoremstyle[
  spaceabove=\topsep,
  spacebelow=\topsep,
  headpunct=,
  numbered=no,
  postheadspace=1ex,
  headfont=\normalfont\bfseries,
  bodyfont=\normalfont\upshape,
]{upright}
\declaretheorem[style=italic,name=Theorem,numbered=yes]{thm}
\declaretheorem[style=italic,name=Theorem,numbered=yes]{athm}

\declaretheorem[style=italic,name=Lemma,numbered=yes,numberlike=thm]{lem}
\declaretheorem[style=italic,name=Proposition,numbered=yes,numberlike=thm]{prop}

\declaretheorem[style=upright,name=Definition,numbered=yes,numberlike=thm]{defn}
\declaretheorem[style=upright,name=Remark,numbered=yes,numberlike=thm]{rmk}

\declaretheorem[style=upright,name=Notation,numbered=yes,numberlike=thm]{notation}

\makeatletter
\renewcommand*{\@seccntformat}[1]{\upshape\csname the#1\endcsname.\hspace{1ex}}
\renewcommand*{\section}{\@startsection{section}{1}{\z@}%
	{2.5ex \@plus 1ex \@minus 0.2ex}%
	{1.5ex \@plus 0.2ex}%
	{\normalfont\Large\bfseries}}
\renewcommand*{\subsection}{\@startsection{subsection}{2}{\z@}%
	{2.5ex \@plus 1ex \@minus 0.2ex}%
	{1.5ex \@plus 0.2ex}%
	{\normalfont\large\bfseries}}
\renewcommand*{\subsubsection}{\@startsection{subsubsection}{3}{\z@}%
	{2.5ex \@plus 1ex \@minus 0.2ex}%
	{1.5ex \@plus 0.2ex}%
	{\normalfont\normalsize\bfseries}}
\newcommand*{\subsubsubsection}{\@startsection{paragraph}{4}{\z@}%
	{2.5ex \@plus 1ex \@minus 0.2ex}%
	{1.5ex \@plus 0.2ex}%
	{\normalfont\normalsize\bfseries}}
\makeatother

\newcommand{\incl}[3][right]%
{%
\draw[<-,>=#1 hook] #2 to ($ #2!0.5!#3 $);
\draw[->] ($ #2!0.5!#3 $) to #3;%
}
\newcommand{\inclusion}[5][right]%
{%
\draw[<-,>=#1 hook] #4 to ($ #4!0.5!#5 $) node[#2,font=\small]{#3};
\draw[->] ($ #4!0.5!#5 $) to #5;%
}

{\begin{compactitem}

}%
{\end{compactitem}}
{\begin{compactitem}[#1]

}%
{\end{compactitem}}
{\begin{compactdesc}

}%
{\end{compactdesc}}

\newcommand{\cE}{\mathcal{E}}

\newcommand{\cR}{\mathcal{R}}

\newcommand{\cU}{\mathcal{U}}

\newcommand{\bD}{\mathbb{D}}

\newcommand{\bQ}{\mathbb{Q}}
\newcommand{\bR}{\mathbb{R}}
\newcommand{\bS}{\mathbb{S}}

\newcommand{\bZ}{\mathbb{Z}}
\newcommand{\LB}{\mathbf{LB}}
\newcommand{\B}{\mathbf{B}}

\renewcommand{\geq}{\geqslant}
\renewcommand{\leq}{\leqslant}
\renewcommand{\footnoterule}{%
  \kern -3pt
  \hrule width \textwidth height 0.4pt
  \kern 2.6pt
}

\definecolor{dred}{rgb}{0.7,0,0}

\definecolor{dgreen}{rgb}{0,0.5,0}

\begin{document}
\title{\Large\bfseries The Burau representations of loop braid groups}
\author{\normalsize Martin Palmer and Arthur Souli{\'e}}
\date{\normalsize\todaysdate}
\maketitle
{
\makeatletter
\renewcommand*{\BHFN@OldMakefntext}{}
\makeatother
\footnotetext{The first author was partially supported by a grant of the Romanian Ministry of Education and Research, CNCS - UEFISCDI, project number \href{https://mdp.ac/pce2020}{PN-III-P4-ID-PCE-2020-2798}, within PNCDI III. The second author was partially supported by the ANR Projects ChroK ANR-16-CE40-0003 and AlMaRe ANR-19-CE40-0001-01.}
}

% 20C12 Integral representations of infinite groups
% 20F36 Braid groups; Artin groups
% 20J05 Homological methods in group theory
% 57M07 Topological methods in group theory
% 57M10 Covering spaces and low-dimensional topology

\begin{abstract}
We give a simple topological construction of the Burau representations of the loop braid groups. There are four versions: defined either on the non-extended or extended loop braid groups, and in each case there is an unreduced and a reduced version. Three are not surprising, and one could easily guess the correct matrices to assign to generators. The fourth is more subtle, and does not seem combinatorially obvious, although it is topologically very natural.
\end{abstract}
%\begin{abstract}
%Nous donnons une construction topologique simple et naturelle des représentations de Burau des groupes de tresses soudées. Il en existe en fait quatre versions : ces représentations peuvent être définies pour les groupes de tresses soudées étendues ou non étendues, et dans ces deux cas, il y a une version réduite et une autre non réduite. Pour trois d'entre elles, d'un point de vue rigoureusement algébrique, on peut aisément déterminer les matrices correspondant aux générateurs des groupes considérés. En revanche, la quatrième est plus subtile et ne semble pas évidente à déterminer d'un strict point de vue combinatoire, alors qu'elle est topologiquement très naturelle à définir.
%\end{abstract}

\section*{Introduction}

Loop braid groups appear in many guises in topology and group theory. They may be seen geometrically as fundamental groups of trivial links in $\bR^3$, diagrammatically as equivalence classes of \emph{welded braids} (closely related to virtual braids and virtual knot theory), algebraically as subgroups of automorphism groups of free groups or combinatorially via explicit group presentations.

Loop braid groups have been studied, from the topological viewpoint of motions of trivial links in $\bR^3$, by Dahm \cite{Dahm1962}, Goldsmith \cite{Goldsmith1981}, Brownstein and Lee \cite{BrownsteinLee1993} and Jensen, McCammond and Meier \cite{JensenMcCammondMeier2006}. In parallel, the \emph{symmetric automorphism groups} and the \emph{braid-permutation groups} (subgroups of $\mathrm{Aut}(F_n)$, which may also be interpreted in terms of welded braids) were studied by McCool \cite{McCool1986},  Collins \cite{Collins1989} and Fenn, Rim\'{a}nyi and Rourke \cite{FennRimanyiRourke1997}. In particular, Fenn, Rim{\'a}nyi and Rourke found a finite presentation of the braid-permutation groups. Later, Baez, Wise and Crans \cite[Theorem 2.2]{BaezWiseCrans2007} showed that their presentation is also a presentation of the group of motions of a trivial link, thus bringing together the two different points of view. Loop braid groups, as well as related groups of ``wickets'', have also been studied more recently by Brendle and Hatcher \cite{BrendleHatcher2013Configurationspacesrings}. For a detailed survey of the many different facets of loop braid groups, see Damiani's survey \cite{Damianijourney}.

The definition that we shall use is the following.

\begin{defn}
Let $\bD^3$ denote the closed unit ball in $\bR^3$ and choose a trivial $n$-component link $U_n$ in its interior. Let $\mathrm{Emb}(U_n,\bD^3)$ denote the set of all smooth embeddings of $U_n$ into the interior of $\bD^3$, equipped with the smooth Whitney topology, and write $\mathrm{Emb}^u(U_n,\bD^3)$ for the path-component containing the inclusion (the superscript ${}^u$ stands for ``unknotted and unlinked''). There is a natural action of the diffeomorphism group $\mathrm{Diff}(U_n) \cong \mathrm{Diff}(\bS^1) \wr \mathfrak{S}_n$ on this space, and we define
\[
\cE(U_n,\bD^3) \coloneqq \mathrm{Emb}^u(U_n,\bD^3) / \mathrm{Diff}(U_n).
\]
The $n$-th \emph{extended loop braid group} is the fundamental group $\LB'_n \coloneqq \pi_1(\cE(U_n,\bD^3))$. Similarly, we define
\[
\cE^+(U_n,\bD^3) \coloneqq \mathrm{Emb}^u(U_n,\bD^3) / \mathrm{Diff}^+(U_n),
\]
where $\mathrm{Diff}^+$ denotes orientation-preserving diffeomorphisms, and the $n$-th (non-extended) \emph{loop braid group} is the fundamental group $\LB_n \coloneqq \pi_1(\cE^+(U_n,\bD^3))$.
\end{defn}

Thus elements of $\LB'_n$ are thought of as loops of $n$-component unlinks in $\bR^3$, and elements of $\LB_n$ are thought of as loops of \emph{oriented} $n$-component unlinks in $\bR^3$. Since $\mathrm{Diff}^+(U_n)$ is an index-$2^n$ subgroup of $\mathrm{Diff}(U_n)$, the natural quotient map
\[
\cE^+(U_n,\bD^3) \relbar\joinrel\twoheadrightarrow \cE(U_n,\bD^3)
\]
is a $2^n$-sheeted covering map, and thus induces an injection
\begin{equation}
\label{eq:LB-to-LBprime}
\LB_n \lhook\joinrel\longrightarrow \LB'_n
\end{equation}
of fundamental groups. Thus we view the (non-extended) loop braid group $\LB_n$ as a subgroup (of index $2^n$) of the extended loop braid group $\LB'_n$.

\paragraph{Generators.}

We fix a basepoint for $\cE^+(U_n,\bD^3)$ where the $n$ circles are arranged on the $xy$-plane in a row from left to right, as pictured in Figure \ref{fig:unlink-complement}. With respect to this basepoint, the loop braid group $\LB_n$ is generated by the elements $\tau_1,\ldots,\tau_{n-1}$ and $\sigma_1,\ldots,\sigma_{n-1}$ illustrated in Figure \ref{fig:loop-braid-generators}. The elements $\tau_i$ and $\sigma_i$ involve only the $i$-th and $(i+1)$-st loops, which are exchanged; for $\tau_i$, no loop passes through the other; for $\sigma_i$, the $i$-th loop passes through the $(i+1)$-st loop. The extended loop braid group $\LB'_n$ is generated by these elements together with the elements $\rho_1,\ldots,\rho_n$, also illustrated in Figure \ref{fig:loop-braid-generators}. For finite presentations of $\LB_n$ and $\LB'_n$ using these generators, see Fenn, Rimanyi and Rourke \cite[\S 1]{FennRimanyiRourke1997} and Brendle and Hatcher \cite[Propositions 3.3 and 3.7]{BrendleHatcher2013Configurationspacesrings}. We note that there are many conflicting conventions for the names of these generators in the literature; in particular, our notation is consistent with \cite{FennRimanyiRourke1997} but inconsistent with \cite{BrendleHatcher2013Configurationspacesrings}.

\begin{figure}[ht]
    \centering
    \includegraphics[scale=0.7]{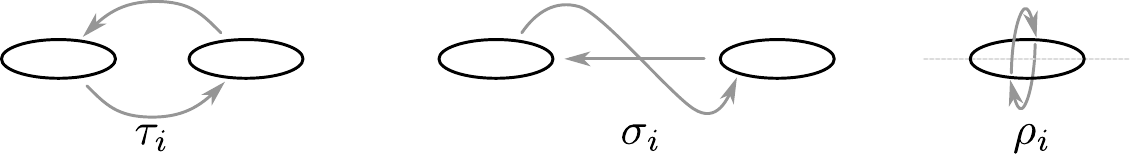}
    \caption{The loop braid group $\LB_n$ is generated by the loops of loops $\tau_1,\ldots,\tau_{n-1}$ and $\sigma_1,\ldots,\sigma_{n-1}$. Together with $\rho_1,\ldots,\rho_n$, these generate the extended loop braid group $\LB'_n$.}
    \label{fig:loop-braid-generators}
\end{figure}

\paragraph{Burau representations of classical braid groups.}

The classical braid groups $\B_n$ are the fundamental groups of the configuration spaces $C_n(\bR^2)$ of points in the plane. One of the oldest interesting representations of $\B_n$ is the \emph{Burau representation}~\cite{burau}
\begin{equation}
\label{eq:Burau-rep}
\B_n \longrightarrow \mathrm{GL}_n(\bZ[t^{\pm 1}]),
\end{equation}
which was defined originally by assigning explicit matrices to the standard generators of $\B_n$, but which is most naturally understood as a homological representation, as follows. The braid group $\B_n$ is naturally isomorphic to the \emph{mapping class group} $\mathrm{MCG}(\bD^2_n) = \pi_0(\mathrm{Diff}_\partial(\bD^2,Q_n))$, the group of isotopy classes of diffeomorphisms of the $2$-disc that act by the identity on its boundary and that preserve a subset $Q_n$ of $n$ points in its interior. In this way, $\B_n$ acts (up to homotopy) on the complement $\bD^2_n = \bD^2 \smallsetminus Q_n$. There is a projection $\pi_1(\bD^2_n) \twoheadrightarrow \bZ$ sending a loop to the sum of its winding numbers around each of the points $Q_n$, and it turns out that the $\B_n$ action on $\bD^2_n$ lifts to the corresponding regular covering space $\pi \colon \widetilde{\bD}^2_n \twoheadrightarrow \bD^2_n$ and commutes with the deck transformations. The induced $\B_n$ action on the first homology $H_1(\widetilde{\bD}^2_n)$ therefore respects its structure as a module over the group-ring of the deck transformation group, $\bZ[\bZ] \cong \bZ[t^{\pm 1}]$, so we obtain a representation
\[
\B_n \longrightarrow \mathrm{Aut}_{\bZ[t^{\pm 1}]} \bigl( H_1 \bigl( \widetilde{\bD}^2_n \bigr) \bigr) .
\]
The homology group $H_1(\widetilde{\bD}^2_n)$ is in fact a free $\bZ[t^{\pm 1}]$-module of rank $n-1$, so choosing a free basis we may rewrite this as
\begin{equation}
\label{eq:reduced-Burau-rep}
\B_n \longrightarrow \mathrm{GL}_{n-1}(\bZ[t^{\pm 1}]).
\end{equation}
This is the \emph{reduced Burau representation}. To obtain the unreduced Burau representation \eqref{eq:Burau-rep}, we consider instead the induced $\B_n$ action on the \emph{relative} first homology $H_1(\widetilde{\bD}^2_n , \pi^{-1}(*))$, where $*$ is a basepoint in the boundary of the disc. This is now a free $\bZ[t^{\pm 1}]$-module of rank $n$, so choosing a free basis we obtain \eqref{eq:Burau-rep}. The canonical map $H_1(\widetilde{\bD}^2_n) \to H_1(\widetilde{\bD}^2_n , \pi^{-1}(*))$ is injective, so the reduced Burau representation \eqref{eq:reduced-Burau-rep} is a subrepresentation of the Burau representation \eqref{eq:Burau-rep}.

Choosing appropriate ordered free generating sets for $H_1(\widetilde{\bD}^2_n)$ and $H_1(\widetilde{\bD}^2_n , \pi^{-1}(*))$ over $\bZ[t^{\pm 1}]$, the representations \eqref{eq:Burau-rep} and \eqref{eq:reduced-Burau-rep} may be written explicitly as
\begin{equation}
\label{eq:classical-Burau-formulas}
\sigma_i \longmapsto I_{i-1} \oplus \left[ \begin{array}{cc}
1-t & 1 \\
t & 0
\end{array} \right] \oplus I_{n-i-1}
\qquad\text{and}\qquad
\sigma_i \longmapsto I_{i-2} \oplus \left[ \begin{array}{ccc}
1 & 0 & 0 \\
t & -t & 1 \\
0 & 0 & 1
\end{array} \right] \oplus I_{n-i-2}
\end{equation}
respectively. We note that the Burau representation is sometimes defined using the transposes of these matrices, such as in \cite{KasselTuraev}, but this is not an essential difference, since the Burau representation is equivalent to its transpose. For more details of these representations, see \cite{KasselTuraev}.

\paragraph{From classical braids to loop braids.}

There is an obvious map
\begin{equation}
\label{eq:map-t}
t \colon C_n(\bR^2) \longrightarrow \cE^+(U_n,\bD^3)
\end{equation}
given by replacing each point in the given configuration with a small circle, oriented positively in $\bR^2$, and then including this unlinked configuration of circles into $\bR^3$. On fundamental groups, this induces a homomorphism $\B_n \to \LB_n$ sending the standard generators of $\B_n$ to the elements $\tau_1,\ldots,\tau_{n-1}$ of $\LB_n$. In particular, this map factors through the projection $\B_n \to \mathfrak{S}_n$ onto the symmetric group on $n$ letters. There is also a more interesting map
\begin{equation}
\label{eq:map-s}
s \colon C_n(\bR^2) \longrightarrow \cE^+(U_n,\bD^3)
\end{equation}
defined as follows. Let us identify $\bR^2$ with the right-hand $xz$-plane (the half where the $x$-coordinate is positive) and the interior of $\bD^3$ with $\bR^3$. Given a configuration of $n$ points in the right-hand $xz$-plane, we produce an $n$-component unlink by rotating the configuration about the $z$-axis, tracing out $n$ circles while doing so, which all lie in planes parallel to the $xy$-plane; see Figure \ref{fig:braid-to-loop-braid}. We orient these circles positively with respect to the parallel copy of the $xy$-plane in which they lie.

\begin{figure}[ht]
    \centering
    \includegraphics[scale=0.8]{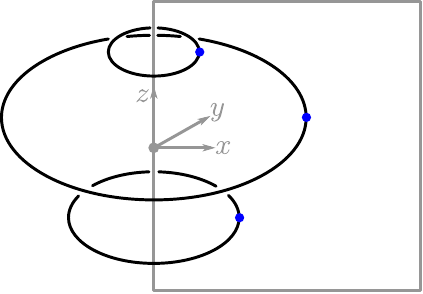}
    \caption{Given a configuration of (blue) points in the right-hand $xz$-plane, we rotate about the $z$-axis as shown to produce a configuration of unlinked circles in $\bR^3$.}
    \label{fig:braid-to-loop-braid}
\end{figure}

To see that the induced homomorphism on fundamental groups is injective, let us take a basepoint in $C_n(\bR^2)$ where the configuration points are arranged in a line along the $x$-axis; this corresponds to a basepoint of $\cE^+(U_n,\bD^3)$ with $n$ concentric circles in the $xy$-plane, centred at the origin. The standard generators of $\B_n$ are sent to loops of the form illustrated on the left-hand side of Figure \ref{fig:sigma-change-basepoint}. Changing the basepoint of $\cE^+(U_n,\bD^3)$ to the one chosen earlier, with \emph{non-concentric} circles on the $xy$-plane, this corresponds to the loop on the right-hand side of Figure \ref{fig:sigma-change-basepoint}, which is the generator $\sigma_i$ of $\LB_n$. By \cite[Proposition 4.3]{BrendleHatcher2013Configurationspacesrings}, the group homomorphism $\B_n \to \LB_n$ sending the standard generators of $\B_n$ to the elements $\sigma_1,\ldots,\sigma_{n-1} \in \LB_n$ is injective. As we have just seen, the map \eqref{eq:map-s} realises this homomorphism at the space level, and so:

\begin{prop}
\label{prop:braids-to-loop-braids}
The map \eqref{eq:map-s} induces an injection on $\pi_1$.
\end{prop}

\begin{rmk}
The map \eqref{eq:map-s} has also been described in \S 6 of \cite{BellingeriBodin}, where its image is called the \emph{configuration space of linear necklaces}. In particular, \cite[Theorem 6.1]{BellingeriBodin} is equivalent to Proposition \ref{prop:braids-to-loop-braids} under this interpretation. A small difference is that the map of \cite{BellingeriBodin} has the space $\cU\cR_n$ as target (see \cite{BrendleHatcher2013Configurationspacesrings} for this notation), whereas \eqref{eq:map-s} has $\cE^+(U_n,\bD^3)$ as target. But \eqref{eq:map-s} factors as
\[
C_n(\bR^2) \longrightarrow \cU\cR_n \longrightarrow \cR_n^+ \longrightarrow \cE^+(U_n,\bD^3),
\]
where the left-hand arrow is the map of \cite{BellingeriBodin}. The middle map is a $\pi_1$-isomorphism by \cite[Proposition 2.3]{BrendleHatcher2013Configurationspacesrings} and the right-hand map is a homotopy equivalence by \cite[Theorem 1]{BrendleHatcher2013Configurationspacesrings}.
\end{rmk}

\begin{figure}[ht]
    \centering
    \includegraphics[scale=0.8]{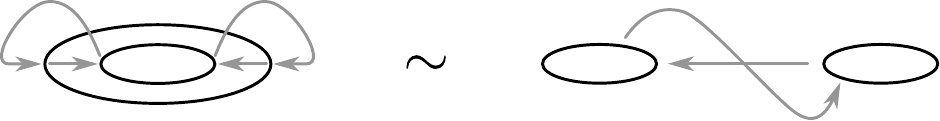}
    \caption{The image of the $i$-th standard generator of $\B_n$ under the map $s_*$ (left) corresponds, by a change of basepoint, to the element $\sigma_i$ of $\LB_n$ (right).}
    \label{fig:sigma-change-basepoint}
\end{figure}

\paragraph{Burau representations of loop braid groups.}

A natural question is whether, and how, one may extend the Burau representations \eqref{eq:Burau-rep} and \eqref{eq:reduced-Burau-rep} along the inclusions
\begin{equation}
\begin{tikzpicture}
[x=1mm,y=1mm]
\node (l) at (0,0) {$\B_n$};
\node (m) at (25,0) {$\LB_n$};
\node (r) at (50,0) {$\LB'_n .$};
\inclusion{above}{$\eqref{eq:map-s}_*$}{(l)}{(m)}
\inclusion{above}{$\eqref{eq:LB-to-LBprime}$}{(m)}{(r)}
\end{tikzpicture}
\end{equation}
The unreduced Burau representation \eqref{eq:Burau-rep} has been extended to $\LB_n$ by Vershinin~\cite{VershininBurau}, using a presentation of $\LB_n$ and by assigning explicit matrices to generators. (More precisely, Vershinin's representation of $\LB_n$ restricts to the \emph{transpose} of the unreduced Burau representation of $\B_n$, according to our conventions.)

Our approach will instead be topological, analogous to the description above of the classical Burau representations of $\B_n$, via the homology of covering spaces. In each case, we will find a natural generating set of the relevant homology group, and calculate the matrix that each standard generator of the loop braid group is sent to.

\begin{notation}
We write $R = \bZ[\bZ] = \bZ[t^{\pm 1}]$ and $S = \bZ[\bZ/2] = \bZ[t^{\pm 1}] / (t^2 - 1)$.
\end{notation}

In this notation, the unreduced and reduced Burau representations are $R$-linear actions $\B_n \curvearrowright R^{\oplus n}$ and $\B_n \curvearrowright R^{\oplus n-1}$ respectively. The case of the non-extended loop braid groups is straightforward:

\begin{athm}
\label{thm:non-extended}
These $R$-linear actions extend to $R$-linear actions $\LB_n \curvearrowright R^{\oplus n}$ and $\LB_n \curvearrowright R^{\oplus n-1}$. \\
Explicit matrices are given in equations \eqref{eq:unreduced-Burau-LB} and \eqref{eq:reduced-Burau-LB} respectively.
\end{athm}

To extend further to the extended loop braid groups is a little more subtle. We must first reduce modulo $t^2 - 1$, in other words tensor $- \otimes_R S$ to obtain $S$-linear actions $\LB_n \curvearrowright S^{\oplus n}$ and $\LB_n \curvearrowright S^{\oplus n-1}$. The unreduced Burau representation then extends directly:

\begin{athm}
\label{thm:extended-unreduced}
The $S$-linear action $\LB_n \curvearrowright S^{\oplus n}$ extends to an $S$-linear action $\LB'_n \curvearrowright S^{\oplus n}$. \\
Explicit matrices are given in equations \eqref{eq:unreduced-Burau-LB} and \eqref{eq:unreduced-Burau-LBprime}.
\end{athm}

The reduced Burau representation does not extend directly; instead:

\begin{athm}
\label{thm:extended-reduced}
The $S$-linear action $\LB_n \curvearrowright S^{\oplus n-1}$ is a subrepresentation of an $S$-linear action $\LB_n \curvearrowright S^{\oplus n-1} \oplus S/(t-1)$, and this extends to an $S$-linear action $\LB'_n \curvearrowright S^{\oplus n-1} \oplus S/(t-1)$. \\
Explicit matrices are given in Table \ref{tab:reduced-Burau-LBprime} on page \pageref{tab:reduced-Burau-LBprime}.
\end{athm}

We emphasise that these extensions of the Burau representations to $\LB_n$ and $\LB'_n$ are precisely those that arise naturally via actions on first homology groups of covering spaces, mirroring the topological construction of the classical Burau representations. As a partial summary, we have
\begin{equation}
\label{eq:Burau-summary}
\centering
\begin{split}
\begin{tikzpicture}
[x=1mm,y=1mm]
\node at (0,15) [font=\small,blue] {$x_1,\ldots,x_{n-1}$};
\node at (30,15) [font=\small,blue] {$x_1,\ldots,x_{n-1}$};
\node at (70,15) [font=\small,blue] {$x_1,\ldots,x_{n-1},y$};
\node at (105,15) [font=\small,blue] {$a_1,\ldots,a_n$};
\node (t1) at (0,10) {$R^{\oplus n-1}$};
\node (t2) at (30,10) {$S^{\oplus n-1}$};
\node (t3) at (70,10) {$S^{\oplus n-1} \oplus S/(t-1)$};
\node (t4) at (105,10) {$S^{\oplus n}$};
\node (b1) at (0,0) {$H_1(\widetilde{\bD}^3_n ; \bZ)$};
\node (b2) at (30,0) {$H_1(\widetilde{\bD}^3_n ; \bZ) \otimes_R S$};
\node (b3) at (70,0) {$H_1(\widehat{\bD}^3_n ; \bZ)$};
\node (b4) at (105,0) {$H_1(\widehat{\bD}^3_n , \{v,tv\} ; \bZ),$};
\node at (0,5) {\rotatebox{90}{$=$}};
\node at (30,5) {\rotatebox{90}{$=$}};
\node at (70,5) {\rotatebox{90}{$=$}};
\node at (105,5) {\rotatebox{90}{$=$}};
\draw[->>] (t1) to (t2);
\incl{(t2)}{(t3)}
\incl{(t3)}{(t4)}
\end{tikzpicture}
\end{split}
\end{equation}
where from left to right we have (1) the reduced Burau representation of $\LB_n$ over $R = \bZ[t^{\pm 1}]$ (Theorem \ref{thm:non-extended}), (2) its reduction modulo $t^2 - 1$ over $S = \bZ[t^{\pm 1}] / (t^2 - 1)$, (3) the inclusion of (2) into the reduced Burau representation of $\LB'_n$ over $S$ (Theorem \ref{thm:extended-reduced}) and (4) its further inclusion into the unreduced Burau representation of $\LB'_n$ (Theorem \ref{thm:extended-unreduced}). In each case, an ordered generating set corresponding to the direct sum decomposition is given in blue.

\begin{rmk}
The matrices of the representations in Theorems \ref{thm:non-extended} and \ref{thm:extended-unreduced} are the ``obvious'' matrices that one may guess by analogy with the matrices for the classical (reduced and unreduced) Burau representations. However, the matrices for the reduced Burau representation of the extended loop braid groups, from Theorem \ref{thm:extended-reduced}, do not seem combinatorially or algebraically obvious. However, they arise very naturally \emph{topologically}. Two additional subtleties in this case are the appearance of torsion in the $S/(t-1)$ summand and the \emph{non-locality} of the matrices for the $\rho_i$ generators.
\end{rmk}

\begin{rmk}
It is stated in Theorem \ref{thm:extended-reduced} that the $\LB_n$-representation $S^{\oplus n-1}$ is a subrepresentation of an $\LB_n$-representation $S^{\oplus n-1} \oplus S/(t-1)$. We remark that it is however \emph{not} a direct summand of the $\LB_n$-representation $S^{\oplus n-1} \oplus S/(t-1)$.
\end{rmk}

See \S\ref{s-properties} for further remarks on reducibility, kernel and other properties of these representations.

\begin{rmk}
\label{rmk:Lawrence-Bigelow}
The Burau representations of the classical braid groups $\B_n$ form the first of an infinite family of \emph{Lawrence-Bigelow} representations \cite{Lawrence1,BigelowHomrep}, and the Burau representations of the loop braid groups $\LB_n$ (or extended loop braid groups $\LB'_n$) may be extended, in more than one way, to an analogous infinite family of representations; see \cite{PalmerSoulie2019}. Explicit bases for some of these representations are computed in \cite[\S 3]{palmersoulie} and further investigation of these families of representations of $\LB_n$ and of $\LB'_n$ will be the subject of forthcoming work.

These representations are particularly interesting as the representation theory of the loop braid groups is in the early stages: so far, few other results are known on extensions of representations of the braid groups to loop braid groups and some of their particular subgroups; see K\'{a}d\'{a}r, Martin, Rowell and Wang \cite{MartinRowell2} and Bellingeri and the second author \cite{Bellingerisoulie}. Furthermore, Damiani, Martin and Rowell \cite{damianimartinrowell} have recently studied a finite dimensional quotient $\mathbf{LH}_{n}$ of the group algebra of $\LB_{n}$, mimicking the braid group/Iwahori-Hecke algebra paradigm. In particular, the unreduced Burau representation of $\LB_{n}$ (Theorem \ref{thm:non-extended}) factors through this quotient algebra $\mathbf{LH}_{n}$; see \cite[\S 3.2]{damianimartinrowell}.
\end{rmk}

\paragraph*{Acknowledgements.}
The authors would like to thank Paolo Bellingeri for his comments on the first version of this paper, in particular for drawing to their attention the reference \cite{BellingeriBodin}.

\section{Action on the homology of covering spaces}

Let $\varphi$ be a diffeomorphism of the $3$-ball $\bD^3$ that restricts to the identity near the boundary. We may restrict $\varphi$ to our chosen unlink $U_n$ in the interior of $\bD^3$ to obtain a new embedding $U_n \hookrightarrow \bD^3$. Since $\varphi$ is isotopic to a diffeomorphism that acts by the identity on $U_n$ (it may be isotoped to act by the identity on a larger and larger collar neighbourhood of the boundary, until this collar neighbourhood contains $U_n$), this new embedding is isotopic to the inclusion, hence an element of $\mathrm{Emb}^u(U_n,\bD^3)$. We therefore have a restriction map
\begin{equation}
\label{eq:restriction-map}
\mathrm{Diff}_\partial(\bD^3) \longrightarrow \mathrm{Emb}^u(U_n,\bD^3),
\end{equation}
where $\mathrm{Diff}_\partial(\bD^3)$ denotes the topological group of diffeomorphisms that are the identity on a neighbourhood of the boundary, equipped with the smooth Whitney topology. The map \eqref{eq:restriction-map} is a locally trivial fibration \cite{Palais1960Localtrivialityof, Cerf1961Topologiedecertains, Lima1963localtrivialityof} and the quotient map
\begin{equation}
\label{eq:quotient-map}
\mathrm{Emb}^u(U_n,\bD^3) \relbar\joinrel\twoheadrightarrow \mathrm{Emb}^u(U_n,\bD^3) / \mathrm{Diff}(U_n) = \cE(U_n,\bD^3)
\end{equation}
is also a locally trivial fibration \cite{BinzFischer1981} (see also \cite[\S 4]{Palmer2018HomologicalstabilitymoduliI} for both of these). Putting together \eqref{eq:restriction-map} and \eqref{eq:quotient-map}, we have a locally trivial fibration
\begin{equation}
\label{eq:fibration}
\mathrm{Diff}_\partial(\bD^3) \longrightarrow \cE(U_n,\bD^3).
\end{equation}
If we modify \eqref{eq:quotient-map} to quotient only by $\mathrm{Diff}^+(U_n)$, it remains a locally trivial fibration, and together with \eqref{eq:restriction-map} we obtain a locally trivial fibration
\begin{equation}
\label{eq:fibration-oriented}
\mathrm{Diff}_\partial(\bD^3) \longrightarrow \cE^+(U_n,\bD^3).
\end{equation}
Together with Hatcher's proof \cite{Hatcher1983} of the Smale conjecture, this implies the following, where $\mathrm{Diff}_\partial(\bD^3,U_n) \leq \mathrm{Diff}_\partial(\bD^3)$ is the subgroup of diffeomorphisms that preserve $U_n$ (setwise) and $\mathrm{Diff}_\partial(\bD^3,U_n^+) \leq \mathrm{Diff}_\partial(\bD^3)$ is the subgroup of diffeomorphisms that preserve $U_n$ and its orientation.

\begin{lem}
\label{lem:mcg-interpretation}
There are isomorphisms
\begin{align*}
\LB'_n = \pi_1(\cE(U_n,\bD^3)) &\cong \pi_0(\mathrm{Diff}_\partial(\bD^3,U_n)) \\
\LB_n = \pi_1(\cE^+(U_n,\bD^3)) &\cong \pi_0(\mathrm{Diff}_\partial(\bD^3,U_n^+)).
\end{align*}
\end{lem}
\begin{proof}
The topological group $\mathrm{Diff}_\partial(\bD^3)$ is contractible, in particular simply-connected, by \cite{Hatcher1983}, and these isomorphisms then follow from the long exact sequences associated to \eqref{eq:fibration} and \eqref{eq:fibration-oriented}.
\end{proof}

\begin{notation}
We will abbreviate $\bD^3_n = \bD^3 \smallsetminus U_n$, where $U_n$ is the $n$-component unlink in the interior of $\bD^3$ chosen previously. See Figure \ref{fig:unlink-complement} for an illustration of a particular choice.
\end{notation}

By the mapping class group interpretation of loop braid groups (Lemma \ref{lem:mcg-interpretation}), the group $\LB'_n$ (and hence also its subgroup $\LB_n$) acts, up to homotopy, on the unlink-complement $\bD^3_n$ by diffeomorphisms (in particular, homeomorphisms).

\begin{rmk}
We will speak of actions of mapping class groups up to homotopy, which induce (strict) actions on homology. An alternative, equivalent viewpoint would be that the corresponding \emph{diffeomorphism group} acts (strictly) at the level of spaces, and then observing that its induced action on homology factors through the mapping class group, since homology groups are discrete.
\end{rmk}

\begin{figure}
    \centering
    \includegraphics[scale=0.5]{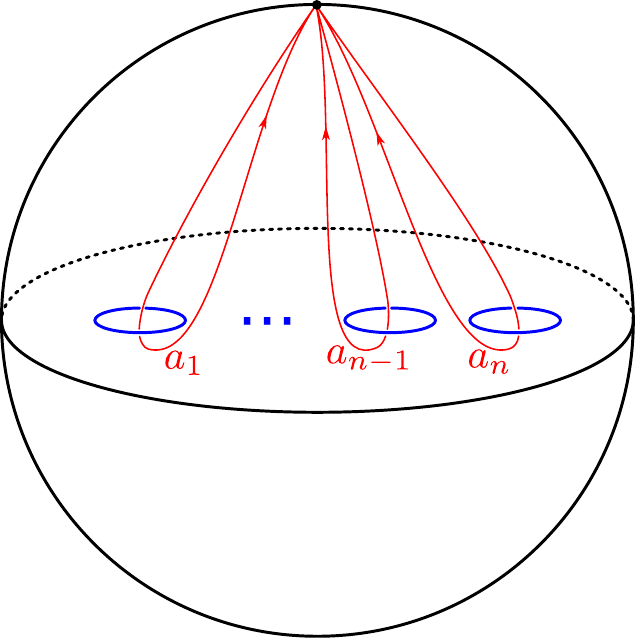}
    \caption{The unlink-complement $\bD^3_n$ with free generators $a_1,\ldots,a_n$ for $\pi_1(\bD^3_n) \cong F_n$.}
    \label{fig:unlink-complement}
\end{figure}

Note that the fundamental group of $\bD^3_n$ is the free group $F_n$ on $n$ generators. This is easy to see: the unlink-complement $\bD^3_n \subseteq \bD^3$ deformation retracts onto a wedge of $n$ circles and $n$ copies of the $2$-sphere. The $n$ circles $a_1,\ldots,a_n$ are shown in Figure \ref{fig:unlink-complement}. Now let
\[
\phi \colon \pi_1(\bD^3_n) \longrightarrow \bZ
\]
be the surjective homomorphism defined by $\phi(a_i) = 1$ for all $i=1,\ldots,n$ and let
\[
\phi' \colon \pi_1(\bD^3_n) \longrightarrow \bZ / 2\bZ
\]
be the composition of $\phi$ with the unique surjection $\bZ \to \bZ/2\bZ$.

\begin{defn}
We denote by $\widetilde{\bD}^3_n$ the regular covering space corresponding to $\mathrm{ker}(\phi)$ and by $\widehat{\bD}^3_n$ the regular covering space corresponding to $\mathrm{ker}(\phi')$. We therefore have regular coverings
\[
\eta \colon \widetilde{\bD}^3_n \longrightarrow \bD^3_n \qquad\text{and}\qquad \eta' \colon \widehat{\bD}^3_n \longrightarrow \bD^3_n
\]
whose deck transformations groups are $\bZ$ and $\bZ/2\bZ$ respectively.
\end{defn}

In general, if a group $G$ acts (up to homotopy) on a based space $X$ and we choose a surjection $\psi \colon \pi_1(X) \twoheadrightarrow Q$ that is invariant under the induced action of $G$ on $\pi_1(X)$, then the $G$-action on $X$ lifts uniquely to the regular covering space corresponding to $\psi$ and commutes with the action of $Q$ by deck transformations.

Let us first take $G=\LB_n$ and $X=\bD_n^3$ with basepoint $* \in \partial \bD^3 = \partial \bD_n^3$. We note that the quotient $\phi$ is invariant under the action of $\LB_n$: for this it suffices to check that each generator $\tau_i , \sigma_i$ of $\LB_n$ sends each generator $a_j$ of $\pi_1(\bD_n^3)$ to an element in $\phi^{-1}(1)$, and this follows since, up to conjugation, $\tau_i$ and $\sigma_i$ simply permute the generators $a_j$. We therefore have an induced action (up to homotopy) of $\LB_n$ on $\widetilde{\bD}^3_n$ commuting with the deck transformation action of $\bZ$. Thus the first integral homology groups
\begin{equation}
\label{eq:LB-reps}
H_1 \bigl( \widetilde{\bD}^3_n ; \bZ \bigr) \qquad\text{and}\qquad H_1 \bigl( \widetilde{\bD}^3_n , \eta^{-1}(*) ; \bZ \bigr)
\end{equation}
are $\bZ[\bZ]$-modules via the deck transformation action, and are $\LB_n$-representations over $\bZ[\bZ]$ via the lifted $\LB_n$-action on $\widetilde{\bD}^3_n$.

\begin{defn}
The $\LB_n$-representations \eqref{eq:LB-reps} are the \emph{reduced} and the \emph{unreduced Burau representations} of loop braid groups over $\bZ[\bZ]=R$.
\end{defn}

Let us now take $G = \LB'_n$ and again $X=\bD_n^3$ with basepoint $* \in \partial \bD^3 = \partial \bD_n^3$. This time $\phi$ is not invariant under the induced action of $\LB'_n$, since, for example, the generator $\rho_i$ sends $a_i \in \phi^{-1}(1)$ to $a_i^{-1} \in \phi^{-1}(-1)$. However, the deeper quotient $\phi'$ (namely $\phi$ reduced mod $2$) is clearly invariant under the action of $\LB'_n$. We therefore have an induced action (up to homotopy) of $\LB'_n$ on $\widehat{\bD}^3_n$ commuting with the deck transformation action of $\bZ/2\bZ$. Thus the first integral homology groups
\begin{equation}
\label{eq:LBprime-reps}
H_1 \bigl( \widehat{\bD}^3_n ; \bZ \bigr) \qquad\text{and}\qquad H_1 \bigl( \widehat{\bD}^3_n , (\eta')^{-1}(*) ; \bZ \bigr)
\end{equation}
are $\bZ[\bZ/2\bZ]$-modules via the deck transformation action, and are $\LB'_n$-representations over $\bZ[\bZ/2\bZ]$ via the lifted $\LB'_n$-action on $\widehat{\bD}^3_n$.

\begin{defn}
The $\LB'_n$-representations \eqref{eq:LBprime-reps} are the \emph{reduced} and the \emph{unreduced Burau representations} of extended loop braid groups over $\bZ[\bZ/2\bZ]=S$.
\end{defn}

Let us make these covering spaces more concrete by building explicit models for each of them. We embed $n$ pairwise disjoint closed $3$-discs into the interior of the unit $3$-disc $\bD^3$ as pictured in Figure \ref{fig:lens-shapes}, so that each little $3$-disc looks like a ``lens shape'' and the union of their equators is precisely the $n$-component unlink that we fixed earlier. Let $\mathring{\bD}^3_n$ denote $\bD^3_n$ minus the interiors of these $n$ little $3$-discs, equivalently, $\bD^3$ minus the interiors and equators of the $n$ little $3$-discs. Also, write $N_i$ for the open northern hemisphere of the boundary of the $i$-th little $3$-disc, and write $S_i$ for the open southern hemisphere of the boundary of the $i$-th little $3$-disc. Now consider
\[
\bZ \times \mathring{\bD}^3_n
\]
and glue $\{j\} \times N_i$ to $\{j-1\} \times S_i$ via the homeomorphism $N_i \cong S_i$ given by reflection in the plane passing through the equator. This is an explicit model for $\widetilde{\bD}^3_n$. Similarly, we may consider
\[
\bZ/2\bZ \times \mathring{\bD}^3_n
\]
and glue as before, where $j$ is now considered mod $2$. This is an explicit model for $\widehat{\bD}^3_n$.

\begin{figure}
    \centering
    \includegraphics[scale=0.5]{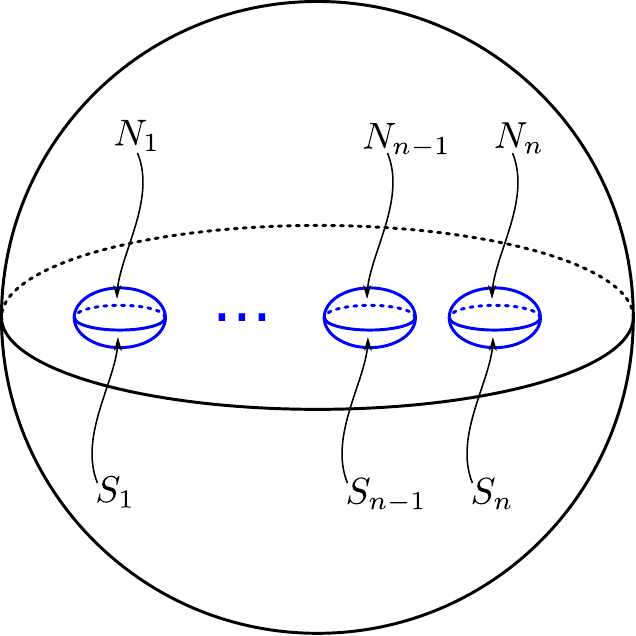}
    \caption{The complement $\mathring{\bD}^3_n$ of the interiors and equators of $n$ closed little $3$-discs (``lens shapes'') in the interior of the closed unit $3$-disc $\bD^3$. The boundary of $\mathring{\bD}^3_n$ decomposes as the disjoint union of $2n+1$ components: $\partial \mathring{\bD}^3_n = \partial \bD^3 \sqcup N_1 \sqcup \ldots \sqcup N_n \sqcup S_1 \sqcup \ldots \sqcup S_n$.}
    \label{fig:lens-shapes}
\end{figure}

\section{Matrices for non-extended loop braid groups}
\label{s-non-extended}

\begin{figure}
    \centering
    \includegraphics[scale=0.8]{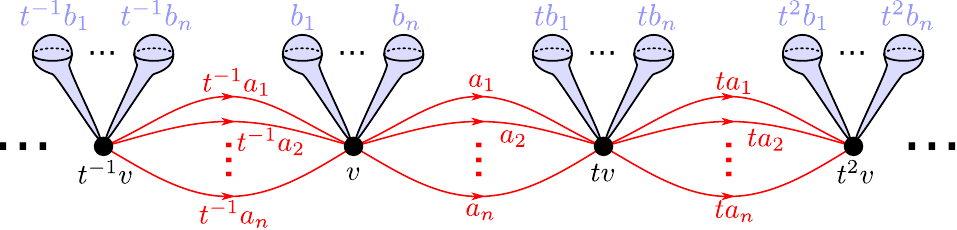}
    \caption{The deformation retract $X$ of the $\bZ$-covering $\widetilde{\bD}^3_n$.}
    \label{fig:infinite-covering}
\end{figure}

We first consider the $\LB_n$-representations \eqref{eq:LB-reps}. The calculations of these representations are unsurprising, but they are a useful warm-up to the slightly more subtle ones in the next section, for the $\LB'_n$-representations \eqref{eq:LBprime-reps}.

\paragraph{The modules.}
As noted above, the unlink-complement $\bD^3_n$ deformation retracts onto a wedge of $n$ circles and $n$ copies of the $2$-sphere. This deformation retraction lifts to a deformation retraction of the covering space $\widetilde{\bD}^3_n$ onto the space pictured in Figure \ref{fig:infinite-covering}. This is an infinite $2$-dimensional cell complex $X$ with vertices indexed by $\bZ$, with exactly $n$ edges between consecutive vertices (and none between non-consecutive vertices) and with exactly $n$ copies of the $2$-sphere wedged onto each vertex. Its fundamental group is freely generated by
$t^k . (a_2 \bar{a}_1) , \ldots\ldots , t^k . (a_n \bar{a}_{n-1})$ for all $k \in \bZ$, where $\bar{a}$ denotes the reverse of a path $a$. Abelianising and writing $\bZ[\bZ] = \bZ[t^{\pm 1}]$, we see that its first homology is freely generated, as a $\bZ[t^{\pm 1}]$-module, by $x_1 \coloneqq a_2 \bar{a}_1 , \ldots\ldots , x_{n-1} \coloneqq a_n \bar{a}_{n-1}$.

The relative homology group $H_1(\widetilde{\bD}^3_n , \eta^{-1}(*) ; \bZ)$ is isomorphic to the first homology of $X$ relative to its set of vertices (since these vertices are fixed by the deformation retraction described above), which is freely generated, as a $\bZ[t^{\pm 1}]$-module, by $a_1,\ldots,a_n$. Summarising, we have natural isomorphisms
\begin{align*}
H_1 \bigl( \widetilde{\bD}^3_n , \eta^{-1}(*) ; \bZ \bigr) &\cong \bZ[t^{\pm 1}]\{ a_1,\ldots,a_n \} \\
H_1 \bigl( \widetilde{\bD}^3_n ; \bZ \bigr) &\cong \bZ[t^{\pm 1}]\{ x_1,\ldots,x_{n-1} \} ,
\end{align*}
and the canonical homomorphism $H_1(\widetilde{\bD}^3_n ; \bZ) \to H_1(\widetilde{\bD}^3_n , \eta^{-1}(*) ; \bZ)$ is given under these identifications by $x_i \mapsto a_{i+1} - a_i$.

\paragraph{The unreduced representation.}

\begin{figure}
    \centering
    \includegraphics[scale=0.7]{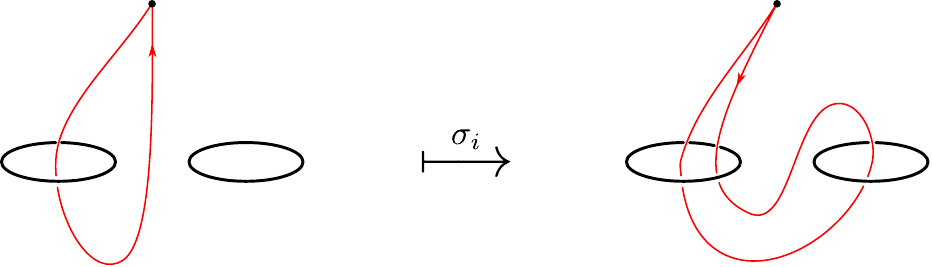}
    \caption{The action of $\sigma_i \in \LB_n$ on the homological generator $a_i$. The right-hand loop may be decomposed (writing from the left to right) as $a_i \cdot ta_{i+1} \cdot t\bar{a}_i$.}
    \label{fig:action-of-sigma}
\end{figure}

It is easy to calculate visually how the $\LB_n$ generators $\tau_i$ and $\sigma_i$ act on the homological generators $a_j$. Clearly $\tau_i$ simply interchanges $a_i$ and $a_{i+1}$. On the other hand, $\sigma_i$ acts by
\[
\sigma_i(a_i) = (1-t).a_i + t.a_{i+1} \qquad\qquad \sigma_i(a_{i+1}) = a_i \qquad\qquad \sigma_i(a_j) = a_j \; (\text{for } j \not\in \{i,i+1\}),
\]
where the first formula comes from the fact that, at the fundamental group level, $\sigma_i$ sends the loop $a_i$ to the loop $a_i \cdot ta_{i+1} \cdot t\bar{a}_i$, which may be read off from Figure \ref{fig:action-of-sigma}. Thus we see that the matrices for the unreduced Burau representation $\LB_n \to \mathrm{GL}_n(\bZ[t^{\pm 1}])$ are given by
\begin{equation}
\label{eq:unreduced-Burau-LB}
\tau_i \longmapsto I_{i-1} \oplus \left[ \begin{array}{cc}
0 & 1 \\
1 & 0
\end{array} \right] \oplus I_{n-i-1}
\qquad\text{and}\qquad
\sigma_i \longmapsto I_{i-1} \oplus \left[ \begin{array}{cc}
1-t & 1 \\
t & 0
\end{array} \right] \oplus I_{n-i-1} .
\end{equation}

\begin{rmk}
These are precisely the transposes of the matrices used in \cite{VershininBurau}. Related to this, we note that, as observed in \cite[Theorem 3.2]{Ibrahim2021}, one may extend the (transpose of the) unreduced Burau representation to the virtual braid group $\mathbf{VB}_n \to \mathrm{GL}_n(\bZ[t^{\pm 1},u^{\pm 1}])$ by
\begin{equation}
\label{eq:unreduced-Burau-VB}
\tau_i \longmapsto I_{i-1} \oplus \left[ \begin{array}{cc}
0 & u^{-1} \\
u & 0
\end{array} \right] \oplus I_{n-i-1}
\qquad\text{and}\qquad
\sigma_i \longmapsto I_{i-1} \oplus \left[ \begin{array}{cc}
1-t & t \\
1 & 0
\end{array} \right] \oplus I_{n-i-1} .
\end{equation}
This factors through the projection $\mathbf{VB}_n \twoheadrightarrow \LB_n$ if one sets $u=1$, but in general it does not. It would be interesting to find a topological construction of this representation, in the sense of the present paper, although it is unclear how this could be done, as we are unaware of a topological interpretation of the virtual braid group as a motion group, analogous to the realisation of the loop braid group as the group of motions of an oriented trivial link in $\bR^3$.
\end{rmk}

\paragraph{The reduced representation.}

Using this computation of the unreduced representation, and the explicit formula ($x_i \mapsto a_{i+1} - a_i$) for the inclusion of the reduced representation into the unreduced one, it is an easy exercise to read off the following explicit formulas for the reduced Burau representation $\LB_n \to \mathrm{GL}_{n-1}(\bZ[t^{\pm 1}])$:
\begin{equation}
\label{eq:reduced-Burau-LB}
\tau_i \longmapsto I_{i-2} \oplus \left[ \begin{array}{ccc}
1 & 0 & 0 \\
1 & -1 & 1 \\
0 & 0 & 1
\end{array} \right] \oplus I_{n-i-2}
\qquad\text{and}\qquad
\sigma_i \longmapsto I_{i-2} \oplus \left[ \begin{array}{ccc}
1 & 0 & 0 \\
t & -t & 1 \\
0 & 0 & 1
\end{array} \right] \oplus I_{n-i-2} .
\end{equation}
If $i=1$ or $i=n-1$, one should ignore the ``$I_{-1}$'' on the left or right, and instead remove the left column and top row, respectively the right column and bottom row, from the displayed matrix.

Observe that, when restricted to the $\sigma_i$ generators, the formulas \eqref{eq:unreduced-Burau-LB} and \eqref{eq:reduced-Burau-LB} are precisely the matrices \eqref{eq:classical-Burau-formulas} defining the unreduced and reduced Burau representations of the classical braid groups. This concludes the proof of Theorem \ref{thm:non-extended}.

\paragraph{Action on the second homology.}

\begin{figure}
    \centering
    \includegraphics[scale=0.6]{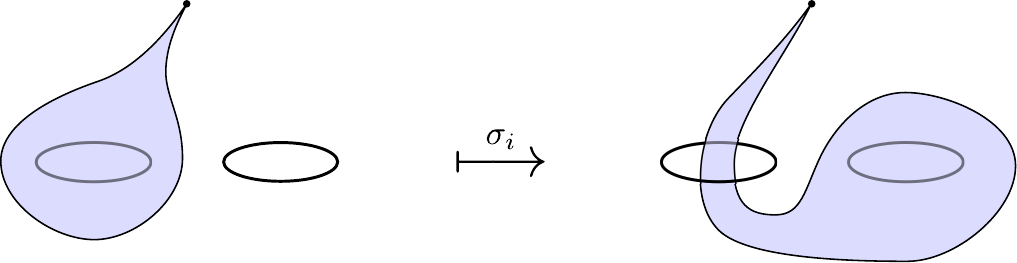}
    \caption{The action of $\sigma_i \in \LB_n$ on the element $b_i \in \pi_2(\widetilde{\bD}_n^3)$. The right-hand side is homotopic to $a_i \cdot b_{i+1}$, where $\cdot$ denotes the action of $\pi_1$ on $\pi_2$.}
    \label{fig:action-of-sigma-on-H2}
\end{figure}

\begin{figure}
    \centering
    \includegraphics[scale=0.6]{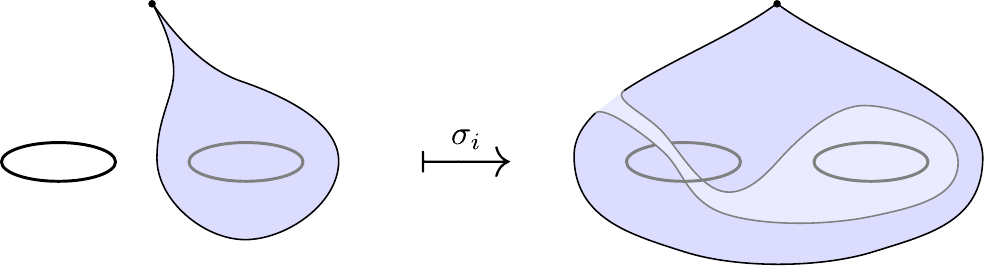}
    \caption{The action of $\sigma_i \in \LB_n$ on the element $b_{i+1} \in \pi_2(\widetilde{\bD}_n^3)$. The right-hand side is homotopic to $b_i + b_{i+1} - a_i \cdot b_{i+1}$, where $\cdot$ denotes the action of $\pi_1$ on $\pi_2$.}
    \label{fig:action-of-sigma-on-H2-a}
\end{figure}

The other non-trivial homology group of the covering space $\widetilde{\bD}_n^3 \simeq X$ is in degree two, where we have $H_2(\widetilde{\bD}_n^3 ; \bZ) \cong \bZ[t^{\pm 1}]\{ b_1,\ldots,b_n \}$, where $b_i$ are illustrated in blue in Figure \ref{fig:infinite-covering}. The generator $\tau_i \in \LB_n$ clearly acts by swapping the homological generators $b_i$ and $b_{i+1}$. The generator $\sigma_i \in \LB_n$ acts as illustrated in Figures \ref{fig:action-of-sigma-on-H2} and \ref{fig:action-of-sigma-on-H2-a}. It sends $b_i$, considered as an element of $\pi_2(\widetilde{\bD}_n^3)$, to $a_i \cdot b_{i+1}$, where $\cdot$ denotes the canonical action of $\pi_1(\widetilde{\bD}_n^3)$ on $\pi_2(\widetilde{\bD}_n^3)$. Inspecting Figure \ref{fig:infinite-covering}, we see that, viewed as an element of $H_2(\widetilde{\bD}_n^3 ; \bZ)$, this is $tb_{i+1}$. Similarly, $\sigma_i$ sends $b_{i+1}$, considered as an element of $\pi_2(\widetilde{\bD}_n^3)$, to $b_i + b_{i+1} - a_i \cdot b_{i+1}$, which is $b_i + (1-t).b_{i+1}$ as an element of $H_2(\widetilde{\bD}_n^3 ; \bZ)$. Thus, with respect to the ordered basis $(b_1,\ldots,b_n)$ of $H_2(\widetilde{\bD}_n^3 ; \bZ)$, this representation $\LB_n \to GL_n(\bZ[t^{\pm 1}])$ is given by
\begin{equation}
\label{eq:representation-on-H2}
\tau_i \longmapsto I_{i-1} \oplus \left[ \begin{array}{cc}
0 & 1 \\
1 & 0
\end{array} \right] \oplus I_{n-i-1}
\qquad\text{and}\qquad
\sigma_i \longmapsto I_{i-1} \oplus \left[ \begin{array}{cc}
0 & 1 \\
t & 1-t
\end{array} \right] \oplus I_{n-i-1} .
\end{equation}
Reversing the ordering, i.e.~using instead the ordered basis $(b_n,\ldots,b_1)$, we obtain
\begin{equation}
\label{eq:representation-on-H2-reversed}
\tau_i \longmapsto I_{i-1} \oplus \left[ \begin{array}{cc}
0 & 1 \\
1 & 0
\end{array} \right] \oplus I_{n-i-1}
\qquad\text{and}\qquad
\sigma_i \longmapsto I_{i-1} \oplus \left[ \begin{array}{cc}
1-t & t \\
1 & 0
\end{array} \right] \oplus I_{n-i-1} ,
\end{equation}
which is the transpose of the unreduced Burau representation \eqref{eq:unreduced-Burau-LB} of $\LB_n$, and agrees with the matrices used in \cite{VershininBurau}.

\section{Matrices for extended loop braid groups}
\label{s-extended}

\begin{figure}
    \centering
    \includegraphics{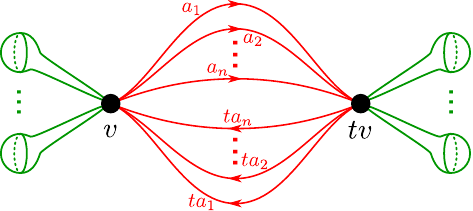}
    \caption{The deformation retract of the double covering $\widehat{\bD}^3_n$ corresponding to quotienting Figure \ref{fig:infinite-covering} by the action of $t^2$.}
    \label{fig:double-covering}
\end{figure}

\paragraph{The modules.}
As in \S\ref{s-non-extended}, the deformation retraction of the unlink-complement $\bD^3_n$ onto a wedge of circles and $2$-spheres lifts to a deformation retraction of its covering $\widehat{\bD}^3_n$ onto the space pictured in Figure \ref{fig:double-covering}. This is a finite $2$-dimensional cell complex with two vertices $\{v,tv\}$, $2n$ edges between them and with $n$ copies of the $2$-sphere wedged onto each vertex. Its fundamental group is freely generated by the $2n-1$ loops
\[
\begin{gathered}
x_1 = a_2 \bar{a}_1, \ldots\ldots , x_{n-1} = a_n \bar{a}_{n-1} \\
tx_1 = ta_2 \, t\bar{a}_1, \ldots\ldots , tx_{n-1} = ta_n \, t\bar{a}_{n-1} \\
y \coloneqq a_n \, ta_n .
\end{gathered}
\]
Hence its first homology $H_1(\widehat{\bD}^3_n ; \bZ)$ is generated, as an abelian group, by the same $2n-1$ loops, viewed as homology classes. The first $2n-2$ of these classes generate a free module of rank $n-1$ over $S = \bZ[\bZ/2] = \bZ[t^{\pm 1}] / (t^2 - 1)$, whereas the last element $y$ generates a summand isomorphic to $\bZ$ viewed as a \emph{trivial} $\bZ[\bZ/2\bZ]$-module, in other words $S / (t-1)$.

The relative homology group $H_1(\widehat{\bD}^3_n , (\eta')^{-1}(*) ; \bZ)$ is isomorphic to the first homology of this complex relative to its vertices $\{t, tv\}$. This is much simpler: it is a free module over $S = \bZ[t^{\pm 1}] / (t^2 - 1)$ of rank $n$, generated by $a_1,\ldots,a_n$. Summarising, we have natural isomorphisms
\begin{align*}
H_1 \bigl( \widehat{\bD}^3_n , (\eta')^{-1}(*) ; \bZ \bigr) &\cong S \{ a_1,\ldots,a_n \} \\
H_1 \bigl( \widetilde{\bD}^3_n ; \bZ \bigr) &\cong S \{ x_1,\ldots,x_{n-1} \} \oplus S/(t-1) \{ y \} .
\end{align*}
The canonical homomorphism $H_1(\widehat{\bD}^3_n ; \bZ) \to H_1(\widehat{\bD}^3_n , (\eta')^{-1}(*) ; \bZ)$ is given under these identifications by $x_i \mapsto a_{i+1} - a_i$ and $y \mapsto (1+t)a_n$. As a matrix, this is:
\begin{equation}
\label{eq:matrix-inclusion}
\left[
\begin{array}{cccccc}
-1 & 0 & 0 &&& \\
1 & -1 & 0 &&& \\
0 & 1 & -1 &&& \\
&&& \smash{\ddots} && \\
&&&& -1 & 0 \\
&&&& 1 & 1+t
\end{array}
\right] .
\end{equation}
Note that this matrix describes an injective homomorphism. (Viewed as an endomorphism of $S^{\oplus n}$, it is of course \emph{not} injective, since $1+t$ is a zero-divisor. But its kernel is precisely the submodule $0^{\oplus n-1} \oplus (t-1)$ of $S^{\oplus n}$ so once we replace the domain with $S^{\oplus n-1} \oplus S/(t-1)$ it becomes injective.)

\paragraph{The unreduced representation.}

The action of the $\LB'_n$ generators $\tau_i$ and $\sigma_i$ on the homological generators $a_j$ is exactly as in \S\ref{s-non-extended}, and given by the matrices \eqref{eq:unreduced-Burau-LB}, considered now over the ring $S = \bZ[t^{\pm 1}] / (t^2 - 1)$ instead of $R = \bZ[t^{\pm 1}]$. The $\LB'_n$ generator $\rho_i$ acts trivially on $a_j$ for $j \neq i$ and sends $a_i$ to $-ta_i$. This last formula comes from the fact that, at the fundamental group level, $\rho_i$ sends the loop $a_i$ to the loop $t\bar{a}_i$, which may be read off from Figure \ref{fig:action-of-rho}. Thus we see that the matrices for the unreduced Burau representation $\LB'_n \to \mathrm{GL}_n(S)$ are given by
\begin{equation}
\label{eq:unreduced-Burau-LBprime}
\eqref{eq:unreduced-Burau-LB} \qquad\text{and}\qquad \rho_i \longmapsto I_{i-1} \oplus \left[
\begin{array}{c}
-t
\end{array}
\right] \oplus I_{n-i} .
\end{equation}
In particular, the restriction of this representation of $\LB'_n$ to the generators $\tau_i$ and $\sigma_i$ (i.e.~to $\LB_n$) is the reduction modulo $t^2$ of the representation \eqref{eq:unreduced-Burau-LB} of Theorem \ref{thm:non-extended}. This concludes the proof of Theorem \ref{thm:extended-unreduced}.

\begin{figure}
    \centering
    \includegraphics[scale=0.7]{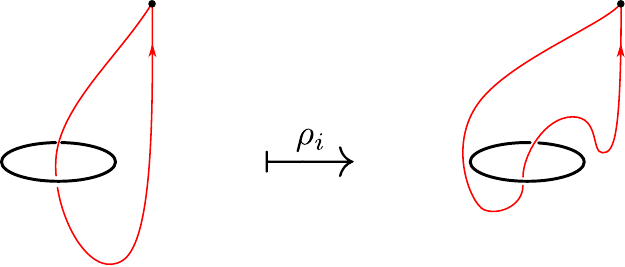}
    \caption{The action of $\rho_i \in \LB'_n$ on the homological generator $a_i$. The right-hand loop is $t\bar{a}_i$. Note that it is not simply $\bar{a}_i$, since this is a path from $tv$ to $v$, whereas $\rho_i(a_i)$ is a path from $v=t^2 v$ to $tv$.}
    \label{fig:action-of-rho}
\end{figure}

\paragraph{The reduced representation.}

It is now a purely algebraic exercise, using the formulas \eqref{eq:unreduced-Burau-LB} and \eqref{eq:unreduced-Burau-LBprime} for the unreduced Burau representation, together with the explicit description \eqref{eq:matrix-inclusion} of the inclusion, to deduce explicit formulas for the reduced Burau representation
\[
\LB'_n \longrightarrow \mathrm{Aut}_S(S^{\oplus n-1} \oplus S/(t-1)).
\]
These are given in Table \ref{tab:reduced-Burau-LBprime}, where we abbreviate $\delta \coloneqq 1+t$. Note that the matrices for the extended generators $\rho_i$ are, in a sense, ``non-local''.

\begin{rmk}
Since these matrices describe automorphisms of $S^{\oplus n-1} \oplus S/(t-1)$, each entry above the bottom row should be considered as an element of $S$, whereas each element of the bottom row should be considered as an element of $S/(t-1) \cong \bZ$. In other words, we set $t=1$ on the bottom row. More precisely, the entries in the bottom row lie in $\mathrm{Hom}_S(S,S/(t-1)) \cong S/(t-1)$, except the bottom-right entry, which lies in $\mathrm{Hom}_S(S/(t-1),S/(t-1)) \cong S/(t-1)$. The entries in the right-hand column (except the bottom-right entry) lie in $\mathrm{Hom}_S(S/(t-1),S) \cong (1+t)S = \delta S \subset S$. (The $S$-modules $S/(t-1)$ and $\delta S$ are abstractly isomorphic, but they are related differently to $S$.)
\end{rmk}

In particular, the restriction of this representation of $\LB'_n$ to the group generators $\tau_i$ and $\sigma_i$ and to the homological generators $x_1,\ldots,x_{n-1}$ is equal to the reduction modulo $t^2$ of the representation \eqref{eq:reduced-Burau-LB} of Theorem \ref{thm:non-extended}. This concludes the proof of Theorem \ref{thm:extended-reduced}.

\begin{rmk}
It is an amusing exercise to verify explicitly that the matrices in Table \ref{tab:reduced-Burau-LBprime} indeed satisfy all of the relations of the extended loop braid group $\LB'_n$, as described for example in \cite[\S 3]{BrendleHatcher2013Configurationspacesrings} or \cite[\S 3]{Damianijourney}. (Warning: the papers \cite{BrendleHatcher2013Configurationspacesrings}, \cite{Damianijourney} and the present paper pairwise disagree on notation for the three families of generators of $\LB'_n$.) One should bear in mind that braid words are written from left to right in \cite{BrendleHatcher2013Configurationspacesrings} and \cite{Damianijourney} (as is usual for composition of loops), whereas matrix multiplication goes from right to left (as is usual for function composition), so in fact the matrices in Table \ref{tab:reduced-Burau-LBprime} satisfy the \emph{opposite} of the relations of $\LB'_n$ described in \cite{BrendleHatcher2013Configurationspacesrings, Damianijourney}. (Technically, this means that we have constructed a representation of the opposite group $(\LB'_n)^{\mathrm{op}}$, but, except for computations, we ignore this subtlety, since this is abstractly isomorphic to $\LB'_n$.) For an explicit verification, using Sage, that the matrices in Table \ref{tab:reduced-Burau-LBprime} satisfy the $40$ relations of the extended loop braid group $\LB'_n$ in the case $n=4$, see the supplementary materials \cite{Supplementary_materials}.
\end{rmk}

\begin{table}[ht]
\centering
\begin{tabular}{|c|c|c|c|}
\hline
& $i=1$ & $2\leq i\leq n-2$ & $i=n-1$ \\ \hline
$\tau_i$ &
$\left[ \begin{array}{cc}
-1 & 1 \\
0 & 1
\end{array} \right] \oplus \mathrm{I}_{n-2}$
&
$\mathrm{I}_{i-2} \oplus \left[ \begin{array}{ccc}
1 & 0 & 0 \\
1 & -1 & 1 \\
0 & 0 & 1
\end{array} \right] \oplus \mathrm{I}_{n-i-1}$
&
$\mathrm{I}_{n-3} \oplus \left[ \begin{array}{ccc}
1 & 0 & 0 \\
1 & -1 & -\delta \\
0 & 0 & 1
\end{array} \right]$
\\ \hline
$\sigma_i$ &
$\left[ \begin{array}{cc}
-t & 1 \\
0 & 1
\end{array} \right] \oplus \mathrm{I}_{n-2}$
&
$\mathrm{I}_{i-2} \oplus \left[ \begin{array}{ccc}
1 & 0 & 0 \\
t & -t & 1 \\
0 & 0 & 1
\end{array} \right] \oplus \mathrm{I}_{n-i-1}$
&
$\mathrm{I}_{n-3} \oplus \left[ \begin{array}{ccc}
1 & 0 & 0 \\
t & -t & -\delta \\
0 & 0 & 1
\end{array} \right]$
\\ \hline\hline
& $i=1$ & $2\leq i\leq n-1$ & $i=n$ \\ \hline
$\rho_i$ &
$\left[ \begin{array}{ccccc}
-t & 0 & \multicolumn{2}{c}{\cdots} & 0 \\
-\delta & \multicolumn{4}{c}{\multirow{4}{*}{\ensuremath{\mathrm{I}_{n-1}}}} \\
\vdots & \multicolumn{4}{c}{} \\
-\delta & \multicolumn{4}{c}{} \\
1 & \multicolumn{4}{c}{}
\end{array} \right]$
&
$\mathrm{I}_{i-2} \oplus \left[ \begin{array}{cccccc}
1 & 0 & 0 & \multicolumn{2}{c}{\cdots} & 0 \\
\delta & -t & 0 & \multicolumn{2}{c}{\cdots} & 0 \\
\delta & -\delta & \multicolumn{4}{c}{\multirow{4}{*}{\ensuremath{\mathrm{I}_{n-i}}}} \\
\vdots & \vdots & \multicolumn{4}{c}{} \\
\delta & -\delta & \multicolumn{4}{c}{} \\
-1 & 1 & \multicolumn{4}{c}{}
\end{array} \right] $
&
$\mathrm{I}_{n-2} \oplus \left[ \begin{array}{cc}
1 & 0 \\
-1 & -1
\end{array} \right]$
\\ \hline
\end{tabular}
\caption{Explicit matrices for the reduced Burau representation of the extended loop braid group $\LB'_n$. Notation: $\delta = 1+t$. All entries lie in $S = \bZ[t^{\pm 1}] / (t^2 - 1)$, except for the bottom row, where they lie in $S / (t-1) \cong \bZ$, in other words we set $t=1$ on the bottom row.}
\label{tab:reduced-Burau-LBprime}
\end{table}

\section{Properties}
\label{s-properties}
\paragraph{Irreducibility.}
The unreduced Burau representations \eqref{eq:unreduced-Burau-LB} of $\LB_n$ and \eqref{eq:unreduced-Burau-LBprime} of $\LB'_n$ are clearly reducible, since they contain the reduced Burau representations \eqref{eq:reduced-Burau-LB} and (Table \ref{tab:reduced-Burau-LBprime}) respectively. The reduced Burau representation \eqref{eq:reduced-Burau-LB} of $\LB_n$ becomes irreducible when we pass to the field of fractions $\bQ(t)$ of $\bZ[t^{\pm 1}]$. This follows because its restriction to the symmetric group $\mathfrak{S}_n \subset \LB_n$ is the standard $(n-1)$-dimensional representation of $\mathfrak{S}_n$, which is irreducible over any field.

On the other hand, for the reduced Burau representation (Table \ref{tab:reduced-Burau-LBprime}) of $\LB'_n$, we cannot directly pass to a field of fractions, since its ground ring $S = \bZ[t^{\pm 1}] / (t^2 - 1)$ is not an integral domain. Instead, we may tensor over $S$ with $\bQ$, setting either $t=-1$ or $t=1$. In the first case, the additional $S/(t-1)$ summand is killed and we obtain an $(n-1)$-dimensional representation of $\LB'_n$ over $\bQ$, which is irreducible, again because its restriction to $\mathfrak{S}_n \subset \LB'_n$ is the standard representation of $\mathfrak{S}_n$. In the second case, we obtain an $n$-dimensional representation:
\begin{equation}
\label{eq:reduced-Burau-LBprime-specialised}
\LB'_n \longrightarrow \mathrm{GL}_n(\bQ).
\end{equation}
\begin{lem}
The representation \eqref{eq:reduced-Burau-LBprime-specialised} of $\LB'_n$ is irreducible.
\end{lem}
\begin{proof}
Suppose that $V \subseteq \bQ^n$ is a non-trivial subrepresentation; we will show that $V = \bQ^n$. Write $v = \alpha_1 x_1 + \cdots + \alpha_{n-1} x_{n-1} + \beta y$.

\emph{Step 1. It suffices to find $v \in V$ with $v \neq 0$ and $\beta = 0$.} \\
First note that $\mathrm{span}_\bQ \{x_1,\ldots,x_{n-1}\}$ is an irreducible subrepresentation (it is irreducible since its restriction to $\mathfrak{S}_n$ is the standard representation of $\mathfrak{S}_n$). Thus $V$ must contain $\mathrm{span}_\bQ \{x_1,\ldots,x_{n-1}\}$. But then it must also contain $x_{n-1} - \rho_n(x_{n-1}) = y$, and so $V = \bQ^n$.

\emph{Step 2. It suffices to find $v \in V$ with $v \neq 0$ and $\alpha_{n-2} - 2\alpha_{n-1} - 2\beta \neq 0$.} \\
For such a $v$, we have $\tau_{n-1}(v) - v = (\alpha_{n-2} - 2\alpha_{n-1} - 2\beta)x_{n-1}$, and we are done by Step 1.

\emph{Step 3. It suffices to find $v \in V$ with $v \neq 0$ and $\alpha_{n-1} + 2\beta \neq 0$.} \\
For such a $v$, we have $\rho_n(v) - v = -(\alpha_{n-1} + 2\beta)y$, and we are done by Step 2.

\emph{Step 4.} Let $v \in V$ be a non-zero vector. By the previous steps, we may assume that its coefficients satisfy $\alpha_{n-2} = \alpha_{n-1} = -2\beta$ and $\beta \neq 0$. We then have $\tau_{n-2}(v) - v = (\alpha_{n-3} + 2\beta)x_{n-2}$, so we are done by Step 1 unless $\alpha_{n-3} = -2\beta$. On the other hand, if $\alpha_{n-3} = -2\beta$, we have $\tau_{n-3}(v) - v = (\alpha_{n-4} + 2\beta)x_{n-3}$, so again we are done by Step 1 unless $\alpha_{n-4} = -2\beta$. Repeating this a further $n-5$ times, we see that we are done unless $\alpha_1 = \alpha_2 = \cdots = \alpha_{n-1} = -2\beta$. But in this case we have $\tau_1(v) - v = 2\beta x_1$, and we are done by Step 1.
\end{proof}

\begin{rmk}
The restriction of the representation \eqref{eq:reduced-Burau-LBprime-specialised} to $\mathfrak{S}_n \subset \LB'_n$ is isomorphic to the regular representation of $\mathfrak{S}_n$. To see this, note that $\mathrm{span}_\bQ \{x_1,\ldots,x_{n-1}\}$ is a subrepresentation isomorphic to the standard representation of $\mathfrak{S}_n$, and the quotient is a trivial $1$-dimensional representation. Thus, by Maschke's theorem, $\eqref{eq:reduced-Burau-LBprime-specialised}|_{\mathfrak{S}_n}$ is isomorphic to the sum of the standard representation and a trivial $1$-dimensional representation, which is isomorphic to the regular representation of $\mathfrak{S}_n$.
\end{rmk}

\paragraph{Kernel.}

The classical (unreduced) Burau representation $\B_n \to \mathrm{GL}_n(\bZ[t^{\pm 1}])$ is known to be faithful for $n \leq 3$ and unfaithful for $n\geq 5$ \cite{Moody1991, LongPaton1993, Bigelow1999}. By contrast, the unreduced Burau representation $\eqref{eq:unreduced-Burau-LB} \colon \LB_n \to \mathrm{GL}_n(\bZ[t^{\pm 1}])$ is unfaithful for all $n \geq 2$, by \cite[Lemmas~4~and~5]{Bardakov}. The unreduced Burau representations for $\LB_n$ and $\LB'_n$ fit together in the commutative square
\begin{equation}
\label{eq:square-of-reduced-Burau-reps}
\begin{tikzcd}
\LB_n \ar[rr,"{\eqref{eq:unreduced-Burau-LB}}"] \ar[d,"\mathrm{incl.}",swap] && \mathrm{GL}_n(R) \ar[d,"{- \otimes_{R} S}"] \\
\LB'_n \ar[rr,"{\eqref{eq:unreduced-Burau-LBprime}}",swap] && \mathrm{GL}_n(S),
\end{tikzcd}
\end{equation}
(recall that $R = \bZ[t^{\pm 1}]$ and $S = R/(t^2 - 1)$) so the unreduced Burau representation $\eqref{eq:unreduced-Burau-LBprime} \colon \LB'_n \to \mathrm{GL}_n(S)$ is also unfaithful for all $n \geq 2$. In fact, the kernel of the composition $\LB_n \to \mathrm{GL}_n(S)$ across the diagonal of \eqref{eq:square-of-reduced-Burau-reps} is larger than the kernel of $\eqref{eq:unreduced-Burau-LB} \colon \LB_n \to \mathrm{GL}_n(R)$ since, for example, $(\tau_1 \sigma_1)^2$ is sent to $\bigl[ \begin{smallmatrix} t^2 & 0 \\ 0 & 1 \end{smallmatrix} \bigr] \oplus \mathrm{I}_{n-2}$.

\paragraph{The transpose of the Burau representation.}

As mentioned in the introduction, the (unreduced) Burau representation of the classical braid group $\B_n$ is equivalent to its transpose. Explicitly, conjugation by the diagonal matrix $\mathrm{Diag}(1,t,\ldots,t^{n-1})$ passes between the Burau representation and its transpose. On the other hand:

\begin{lem}
The unreduced Burau representation \eqref{eq:unreduced-Burau-LB} of the loop braid group $\LB_n$ is not equivalent to its transpose.
\end{lem}
\begin{proof}
Let $\eqref{eq:unreduced-Burau-LB} = \mathrm{Bur}$. The aim is to show that there is no invertible matrix $M \in \mathrm{GL}_n(\bZ[t^{\pm 1}])$ such that $\mathrm{Bur}(g) M = M \mathrm{Bur}(g)^{\mathrm{t}}$ for all $g \in \LB_n$, where $(-)^{\mathrm{t}}$ denotes the transpose of a matrix. For $n=2$, the equations $\mathrm{Bur}(\sigma_1) M = M \mathrm{Bur}(\sigma_1)^{\mathrm{t}}$ and $\mathrm{Bur}(\tau_1) M = M \mathrm{Bur}(\tau_1)^{\mathrm{t}}$ imply that $M$ is of the form $\bigl[ \begin{smallmatrix} a & -a \\ -a & a \end{smallmatrix} \bigr]$, which is not invertible. For $n\geq 3$, we will prove by induction on $n$ the statement that the only matrix $M$ satisfying $\mathrm{Bur}(g) M = M \mathrm{Bur}(g)^{\mathrm{t}}$ for all $g \in \LB_n$ is the zero matrix. We begin with the base case $n=3$. Applying the argument for $n=2$ to the top-left and bottom-right $2 \times 2$ blocks of $M$, we deduce that $M$ must be of the form $\Bigl[ \begin{smallmatrix} a & -a & b \\ -a & a & -a \\ c & -a & a \end{smallmatrix} \Bigr]$. Given this, the equation $\mathrm{Bur}(\tau_1) M = M \mathrm{Bur}(\tau_1)^{\mathrm{t}}$ then implies that $b=c=-a$, and the equation $\mathrm{Bur}(\sigma_1) M = M \mathrm{Bur}(\sigma_1)^{\mathrm{t}}$ implies that $a = at$, thus $a=0$ and $M$ is the zero matrix. For the inductive step, we may apply the inductive hypothesis to the top-left and bottom right $(n-1) \times (n-1)$ blocks of $M$ to see that the entries of $M$ are all zero except possibly the top-right and bottom-left entries. But then the equation $\mathrm{Bur}(\tau_1) M = M \mathrm{Bur}(\tau_1)^{\mathrm{t}}$ implies that these are zero too.
\end{proof}

Thus the representations \eqref{eq:unreduced-Burau-LB} and $\eqref{eq:unreduced-Burau-LB}^{\mathrm{tr}}$ are non-equivalent representations of $\LB_n$. However, we have constructed both of these topologically: \eqref{eq:unreduced-Burau-LB} is the action of $\LB_n$ on the first homology group $H_1(\widetilde{\bD}_n^3 , \eta^{-1}(*) ; \bZ)$ and $\eqref{eq:unreduced-Burau-LB}^{\mathrm{tr}}$ is the action of $\LB_n$ on the second homology group $H_2(\widetilde{\bD}_n^3 ; \bZ)$, as shown at the end of \S\ref{s-non-extended}.

\phantomsection
\addcontentsline{toc}{section}{References}
\renewcommand{\bibfont}{\normalfont\small}
\setlength{\bibitemsep}{0pt}
\printbibliography

\noindent {Martin Palmer, \itshape Institutul de Matematic\u{a} Simion Stoilow al Academiei Rom{\^a}ne, 21 Calea Griviței, 010702 București, Rom{\^a}nia}. \noindent {Email address: \tt mpanghel@imar.ro}

\noindent {Arthur Soulié, \itshape University of Glasgow, School of Mathematics and Statistics, 132 University Pl, Glasgow G12 8TA, United Kingdom}. \noindent {Email address: \tt artsou@hotmail.fr}
\end{document}